\documentclass{compositio}

\usepackage{amsmath}
\usepackage{amsthm,amsfonts,amssymb,mathrsfs}
\usepackage{epic,eepic}
\usepackage{yfonts}
\usepackage{paralist,enumerate}
\usepackage[all]{xy}

\newtheorem{theorem}{Theorem}[section]

\newtheorem{lemma}[theorem]{Lemma}
\newtheorem{corollary}[theorem]{Corollary}

\newtheorem*{Var}{Varchenko's conjecture}
\newtheorem{Theorem}{Theorem}

\theoremstyle{definition}
\newtheorem{definition}[theorem]{Definition}

\theoremstyle{remark}
\newtheorem{remark}[theorem]{Remark}
\newtheorem{example}[theorem]{Example}

\begin{document}

\title[The maximum likelihood degree of a very affine variety]{The maximum likelihood degree of a very affine variety}
\author{June Huh}
\email{junehuh@umich.edu}
\address{Department of Mathematics, University of Michigan\\
Ann Arbor, MI 48109\\  USA}

\begin{abstract}
We show that the maximum likelihood degree of a smooth very affine variety is equal to the signed topological Euler characteristic. This generalizes Orlik and Terao's solution to Varchenko's conjecture on complements of hyperplane arrangements to smooth very affine varieties. For very affine varieties satisfying a genericity condition at infinity, the result is further strengthened to relate the variety of critical points to the Chern-Schwartz-MacPherson class. The strengthened version recovers the geometric deletion-restriction formula of Denham et al.\ for arrangement complements, and generalizes Kouchnirenko's theorem on the Newton polytope for nondegenerate hypersurfaces. 
\end{abstract}

\maketitle

\section{Introduction}

Maximum likelihood estimation in statistics leads to the problem of finding critical points of a product of powers of polynomials on an algebraic variety \cite[Section 3.3]{Pachter-Sturmfels}. When the polynomials and the variety are linear and defined over the real numbers, the number of critical points is the number of bounded regions in the corresponding arrangement of hyperplanes.

Studying Bethe vectors in statistical mechanics, Varchenko conjectured a combinatorial formula for the number of critical points for complex hyperplane arrangements \cite{Varchenko}.  Let $U$ be the complement of $n$ hyperplanes in $\mathbb{C}^r$ defined by the linear functions $f_1,\ldots,f_n$. The master function $\varphi_{\bf u}=\prod_{i=1}^n f_i^{u_i}$, where the exponents $u_i$ are integral parameters, is a holomorphic function on $U$. We assume that the affine hyperplane arrangement is \emph{essential}, meaning that the lowest-dimensional intersections of the hyperplanes are isolated points. 

\begin{Var}
If the hyperplane arrangement is essential and the exponents $u_i$ are sufficiently general, then the following hold.
\begin{enumerate}
\item $\varphi_{\bf u}$ has only finitely many critical points in $U$.
\item All critical points of $\varphi_{\bf u}$ are nondegenerate.
\item The number of critical points is equal to the signed Euler characteristic $(-1)^r \chi(U)$.
\end{enumerate}
\end{Var}

The conjecture was proved by Varchenko in the case where the hyperplanes are defined over the real numbers \cite{Varchenko}, and by Orlik and Terao in general \cite{Orlik-Terao}.
Subsequent works of Silvotti and Damon extended this result to some nonlinear arrangements \cite{Damon1,Damon2,Silvotti}. The assumption made on the arrangement is certainly necessary, for there are arrangements violating the inequality $(-1)^r \chi(U) \ge 0$.

The principal aim of this paper is to generalize the theorem of Orlik and Terao. The generalization is pursued in two directions. In Theorem \ref{main}, we obtain the same conclusion for a wider class of affine varieties than complements of essential arrangements; in Theorem \ref{proCSM}, we recover the whole characteristic class from the critical points instead of the topological Euler characteristic. A connection to Kouchnirenko's theorem on the relation between the Newton polytope and the Euler characteristic is pointed out in Section \ref{Kushnirenko}.

The above extensions are motivated by the problem of maximum likelihood estimation in algebraic statistics. Recall that an irreducible algebraic variety is said to be \emph{very affine} if it is isomorphic to a closed subvariety of an algebraic torus. Very affine varieties have recently received considerable attention due to their central role in tropical geometry \cite{EKL,Speyer,Tevelev}. The complement of an affine hyperplane arrangement is affine, and it is very affine if and only if the hyperplane arrangement is essential. Any complement of an affine hyperplane arrangement is of the form $U \times \mathbb{C}^k$, where $U$ is the complement of an essential arrangement.

In view of maximum likelihood estimation, very affine varieties are the natural class of objects generalizing complements of hyperplane arrangements. Consider the projective space with the homogeneous coordinates $p_1,\ldots,p_n$, where the coordinate $p_i$ represents the probability of the $i$-th event. An implicit statistical model is a closed subvariety $V \subseteq \mathbb{P}^{n-1}$. The data comes in the form of nonnegative integers $u_1,\ldots,u_n$, where $u_i$ is the number of times the $i$-th event was observed. 

In order to find the values of $p_i$ on $V$ which best explain the given data $u_i$, one finds critical points of the likelihood function
\[
L(p_1\ldots,p_n)= p_1^{u_1}  \cdots  p_n^{u_n} / (p_1+\cdots+p_n)^{u_1+\cdots+u_n}.
\]
Statistical computations are typically done in the affine chart defined by the nonvanishing of $p_1+\cdots+p_n$, where the sum can be set equal to $1$ and the denominator of $L$ can be ignored.  The maximum likelihood degree of the model is defined to be the number of complex critical points of the restriction of $L$ to the projective variety $V$, where we only count critical points that are not poles or zeros of $L$, and $u_1,\ldots,u_n$ are assumed to be sufficiently general \cite{Hosten-Khetan-Sturmfels}. In other words, the maximum likelihood degree is the number of critical points of the likelihood function on the very affine variety 
\[
U:=\big\{x \in V \mid p_1 \cdots p_n (p_1+\cdots+p_n) \neq 0\big\}.
\]

\subsection{Varchenko's conjecture for very affine varieties}

We extend the theorem of Orlik and Terao to smooth very affine varieties. Let $U$ be a smooth very affine variety of dimension $r$. Choose a closed embedding
\[
f: U \longrightarrow (\mathbb{C}^*)^n, \qquad f=(f_1,\ldots,f_n).
\]
The master function $\varphi_{\bf u}=\prod_{i=1}^n f_i^{u_i}$, where the exponents $u_i$ are integral parameters, is a holomorphic function on $U$. The \emph{maximum likelihood degree} of $U$ is defined to be the number of critical points of the master function with sufficiently general exponents $u_i$. 

\begin{Theorem}\label{main}
If the exponents $u_i$ are sufficiently general, then the following hold.
\begin{enumerate}
\item $\varphi_{\bf u}$ has only finitely many critical points in $U$.
\item All critical points of $\varphi_{\bf u}$ are nondegenerate.
\item The number of critical points is equal to the signed Euler characteristic $(-1)^r \chi(U)$.
\end{enumerate}
\end{Theorem}

More precisely, there is a nonzero polynomial $F$ such that the assertions are valid for $u_1,\ldots,u_n$ with $F(u_1,\ldots,u_n) \neq 0$.

Theorem \ref{main} shows that, for instance, the conclusions of \cite[Theorem 20]{Catanese-Hosten-Khetan-Sturmfels} and \cite[Corollary 6]{Damon1} hold for smooth very affine varieties without further assumptions. This has a few immediate corollaries that might be of interest in algebraic geometry and algebraic statistics. First, the maximum likelihood degree does not depend on the embedding of $U$ into an algebraic torus. Second, the maximum likelihood degree satisfies the deletion-restriction formula as in the case of a linear model. Third, the sign of the Euler characteristic of a smooth very affine variety depends only on the parity of its dimension.

\subsection{A geometric formula for the CSM class}

The theorem of Orlik and Terao can be further generalized to very affine varieties which admit a good tropical compactification in the sense of Tevelev \cite{Tevelev}. See Definition \ref{schon} for \emph{sch\"on} very affine varieties.
For example, the complement of an essential hyperplane arrangement is sch\"on, 
and the hypersurface defined by a sufficiently general Laurent polynomial (with respect to its Newton polytope) is sch\"on.
The open subset of the Grassmannian $\text{Gr}(2,n)$ given by nonvanishing of all Pl\"ucker coordinates is another sch\"on very affine variety, which is of interest in algebraic statistics \cite{Speyer-Sturmfels}.

The generalization is formulated in terms of the variety of critical points of $U$, the totality of critical points of all possible (multivalued) master functions for $U$. More precisely, given a compactification $\overline{U}$ of $U$, the variety of critical points $\mathfrak{X}(U)$ is defined to be the closure
\[
\mathfrak{X}(U)=\overline{\mathfrak{X}^\circ(U)} \subseteq \overline{U} \times \mathbb{P}^{n-1} \quad\text{of}\ \ \mathfrak{X}^\circ(U)=\Bigg\{\sum_{i=1}^n u_i \cdot \text{dlog}(f_i)(x)=0 \Bigg\} \subseteq U \times \mathbb{P}^{n-1},
\]
where $\mathbb{P}^{n-1}$ is the projective space with the homogeneous coordinates $u_1,\ldots,u_n$. The variety of critical points has been studied previously in the context of hyperplane arrangements \cite{CDFV,Denham-Garrousian-Schulze}. See Section \ref{proof} for a detailed construction in the general setting.

We relate the variety of critical points to the Chern-Schwartz-MacPherson class of \cite{MacPherson}.
Let $\mathbb{T}_U$ be the intrinsic torus of $U$, an algebraic torus containing $U$ whose character lattice is the group of nonvanishing regular functions on $U$ modulo nonzero constants. We compactify the intrinsic torus by the projective space $\mathbb{P}^n$, where $n$ is the dimension of $\mathbb{T}_U$.

\begin{Theorem}\label{proCSM}
Suppose that $U$ is an $r$-dimensional very affine variety which is not isomorphic to a torus. If $U$ is sch\"on, then
\[
\big[\mathfrak{X}(U)\big] =
\sum_{i=0}^{r} v_i \big[ \mathbb{P}^{r-i} \times \mathbb{P}^{n-1-r+i}\big] \in A_{*}(\mathbb{P}^n \times \mathbb{P}^{n-1}),
\]
where
\[
c_{SM}(\mathbf{1}_U)=\sum_{i=0}^{r} (-1)^i v_i \hspace{0.7mm} [\mathbb{P}^{r-i}] \in A_*(\mathbb{P}^n).
\]
\end{Theorem}

Theorem \ref{main} is recovered by considering the number of points in a general fiber of the second projection from $\mathfrak{X}(U)$, which is the maximum likelihood degree
\[
v_r=(-1)^r\int c_{SM}(\mathbf{1}_U)= (-1)^r\chi(U).
\]

When $U$ is the complement of an essential hyperplane arrangement $\mathcal{A}$ and $\mathbb{P}^n$ is the usual compactification of $\mathbb{T}_U$ defined by the ratios of homogeneous coordinates $z_1/z_0,\ldots,z_n/z_0$, Theorem \ref{proCSM} specializes to the geometric formula for the characteristic polynomial of Denham et al. \cite[Theorem 1.1]{Denham-Garrousian-Schulze}:
\[
\chi_\mathcal{A}(q+1)=\sum_{i=0}^{r} (-1)^i v_i \hspace{0.7mm} q^{r-i}.
\]
The formula  is used in \cite{Huh2} to verify Dawson's conjecture on the logarithmic concavity of the $h$-vector of a matroid complex, for matroids representable over a field of characteristic zero \cite{Dawson}. Other implications of the geometric formula are collected in Remark \ref{logconcave}.

\subsection{A generalization of Kouchnirenko's theorem}

A sch\"on hypersurface in an algebraic torus is defined by a Laurent polynomial which is nondegenerate in the sense of Kouchnirenko \cite{Kouchnirenko}. We generalize Kouchnirenko's theorem equating the Euler characteristic with the signed volume of the Newton polytope, in the setting of Theorem \ref{proCSM}. We hope that the approach of the present paper clarifies an analogy noted in \cite[Remarks (e)]{Varchenko}, where Varchenko asks for a connection between Kouchnirenko's theorem and the conjecture stated in the introduction.

Let $g$ be a Laurent polynomial in $n$ variables with the Newton polytope $\Delta_g$, and denote the corresponding hypersurface by
\[
U=\{g=0\} \subseteq (\mathbb{C}^*)^n.
\]
Fix the open embedding $(\mathbb{C}^*)^n \subset \mathbb{P}^n$ defined by the ratios of homogeneous coordinates $z_1/z_0,\ldots,z_n/z_0$.

We follow the convention of \cite[Chapter 7]{Cox-Little-Oshea} and write $\text{MV}_n$ for the $n$-dimensional mixed volume. For example, the $n$-dimensional standard simplex $\Delta$ in $\mathbb{R}^n$ has the unit volume $1$.

\begin{Theorem}\label{Kush}
Let $g$ be a nonzero Laurent polynomial in $n=r+1$ variables with
\[
c_{SM}(\mathbf{1}_{U}) = \sum_{i=0}^{r} (-1)^{i} v_i \hspace{0.7mm} [\mathbb{P}^{r-i}] \in A_*(\mathbb{P}^n).
\]
If $g$ is nondegenerate, then
\[
v_i=\text{MV}_n(\underbrace{\Delta,\ldots,\Delta}_{r-i},\underbrace{\Delta_g,\ldots,\Delta_g}_{i+1}) \quad \text{for} \ i=0,\ldots,r.
\]
In particular, the maximum likelihood degree of $U$ is equal to the normalized volume
\[
v_r=(-1)^{r} \int c_{SM}(\mathbf{1}_{U})= \text{Volume}(\Delta_g).
\]
\end{Theorem}

Theorem \ref{Kush} has applications not covered by Kouchnirenko's theorem. In particular, we get an explicit formula for the degree of the gradient map of a homogeneous polynomial in terms of the Newton polytope; see Corollary \ref{ProjectiveDegree}. This shows that many delicate examples discovered in classical projective geometry have a rather simple combinatorial origin. 

As an example, we find irreducible homaloidal projective hypersurfaces of given degree $d \ge 3$ and ambient dimension $n \ge 3$, improving upon previous constructions in \cite{CRS,Fassarella-Medeiros}; see Example \ref{homaloidal}.

\subsection{Organization}

We provide a brief overview of the paper.

Section \ref{proof} is devoted to the proof of Theorem \ref{main}. Along the way we construct the variety of critical points and describe its basic properties. 

Section \ref{additivity} introduces the deletion-restriction for hyperplane arrangements, extending that to very affine varieties. A brief introduction to the Chern-Schwartz-MacPherson class is given, and Theorem \ref{proCSM} is proved. 

Section \ref{Kushnirenko} focuses on the maximum likelihood degree of nondegenerate hypersurfaces in algebraic tori. Applications of Theorem \ref{Kush} to the geometry of projective hypersurfaces are given.

\section{Proof of Theorem \ref{main}}\label{proof}

\subsection{The Gauss map of very affine varieties}\label{GaussMap}

An important role will be played by the Gauss map of a very affine variety in its intrinsic torus. Let $U$ be a smooth very affine variety of dimension $r$. Choose a closed embedding
\[
f: U \longrightarrow (\mathbb{C}^*)^n, \qquad f=(f_1,\ldots,f_n).
\]
By a theorem of Samuel (see \cite{Samuel}), the group of invertible regular functions $M_U:=\mathbb{C}[U]^*/\mathbb{C}^*$ is a finitely generated free abelian group. Therefore one may choose $f_i$ to form a basis of $M_U$. In this case, $f$ is a closed embedding of $U$ into the intrinsic torus $\mathbb{T}_U$ with the character lattice $M_U$
\[
f: U \longrightarrow \mathbb{T}_U.
\]
Any morphism from $U$ to an algebraic torus $\mathbb{T}$ is a composition of $f$ with a homomorphism $\mathbb{T}_U \to \mathbb{T}$. The Gauss map of $U$ is defined by the pushforward of $f$ followed by left-translation to the identity; that is,
\[
T_xU \longrightarrow T_{f(x)} \mathbb{T}_U \longrightarrow T_{{\bf 1}} \mathbb{T}_U \quad \text{for} \ x \in U.
\]
In coordinates, the first map is represented by the Jacobian matrix 
\[
\Bigg(\frac{\partial f_i}{\partial x_j}\Bigg), \quad 1 \le i \le n, \quad 1 \le j \le r,
\] 
and the second map is represented by the diagonal matrix with diagonal entries $1/f_i(x)$. The composition of the two is the logarithmic Jacobian matrix 
\[
\Bigg(\frac{\partial \log f_i}{\partial x_j}\Bigg), \quad 1 \le i \le n, \quad 1 \le j \le r.
\] 
This defines the Gauss map from $U$ to the Grassmannian of $T_\mathbf{1}\mathbb{T}_U$:
\[
U \longrightarrow \text{Gr}_r(T_\mathbf{1}\mathbb{T}_U), \qquad x \longmapsto T_xU \subseteq T_\mathbf{1}\mathbb{T}_U.
\]
Let $\Omega^1_U$ be the sheaf of differential one-forms on $U$. Consider the complex vector space 
\[
W:=M_U \otimes_\mathbb{Z} \mathbb{C}.
\]
The dependence of $W$ on $U$ will often be omitted from the notation. We write $W_U$ for (the sheaf sections of) the trivial vector bundle over $U$ with the fiber $W$. There is a vector bundle homomorphism $\Phi$, defined by the evaluation of the logarithmic differential forms as follows:
\[
\Phi: W_U \longrightarrow \Omega^1_U, \qquad \Big(\prod_{i=1}^n f_i^{u_i},x\Big) \longmapsto \sum_{i=1}^n u_i \cdot \text{dlog}(f_i)(x).
\]
At a point $x \in U$, the linear map $\Phi(x)$ between the fibers is dual to the injective linear map considered above,
\[
T_xU \longrightarrow T_{f(x)} \mathbb{T}_U \longrightarrow T_{{\bf 1}} \mathbb{T}_U.
\]
Therefore $\Phi$ is surjective and $\ker \Phi$ is a vector bundle over $U$.

\subsection{The variety of critical points}\label{VarietyofCriticalPoints} The inclusion of $\ker \Phi$ into $W_U$ defines a closed immersion between the projective bundles
\[
\mathfrak{X}^\circ(U):=\text{Proj}\big(\text{Sym}(\ker \Phi ^\vee)\big) \longrightarrow \text{Proj}\big(\text{Sym}(W_U^\vee)\big) \simeq U \times \mathbb{P}(W).
\]
Note that the following conditions are equivalent:

\begin{enumerate}[1.]
\item $\Phi$ is injective.
\item $U$ is isomorphic to a torus.
\item $\mathfrak{X}^\circ(U)$ is empty.
\end{enumerate}
If $\mathfrak{X}^\circ(U)$ is not empty, then $\mathfrak{X}^\circ(U)$ is a projective bundle over $U$ of dimension equal to that of $\mathbb{P}(W)$, defined by the equation
\[
\sum_{i=1}^n u_i \cdot \text{dlog}(f_i)(x)=0
\]
where $u_1,\ldots,u_n$ are the homogeneous coordinates of $\mathbb{P}(W)$. In short, $\mathfrak{X}^\circ(U)$ is the set of critical points of all possible (multivalued) master functions.

\begin{definition}\label{VC}
Given a compactification $V$ of $U$, the \emph{variety of critical points} of $U$ is defined to be the closure 
\[
\mathfrak{X}_V(U):=\overline{\mathfrak{X}^\circ(U)} \subseteq V \times \mathbb{P}(W).
\]
We denote the variety of critical points by $\mathfrak{X}(U)$ when there is no danger of confusion.
\end{definition}

The variety of critical points is irreducible by construction. When $U$ is the complement of an essential arrangement of hyperplanes, $\mathfrak{X}(U)$ is the variety of critical points previously considered in the context of hyperplane arrangements \cite{CDFV,Denham-Garrousian-Schulze}. This variety has its origin in \cite[Proposition 4.1]{Orlik-Terao}.

We record here the following basic compatibility: If $V_1 \to V_2$ is a morphism between two compactifications of $U$ which is the identity on $U$, then the class of $\mathfrak{X}_{V_1}(U)$ maps to the class of $\mathfrak{X}_{V_2}(U)$ under the induced map between the Chow groups
\[
A_*\big(V_1 \times \mathbb{P}(W)\big) \longrightarrow A_*\big(V_2 \times \mathbb{P}(W)\big), \qquad \big[\mathfrak{X}_{V_1}(U)\big] \longmapsto \big[\mathfrak{X}_{V_2}(U)\big].
\]

\begin{remark}
We point out a technical difference between Definition \ref{VC} and the variety of critical points in the cited literature. This remark is intended for readers familiar with \cite{CDFV,Denham-Garrousian-Schulze} and is independent of the rest of the paper.

Let $U$ be the complement of $\mathcal{A}$, an essential arrangement of $n$ hyperplanes in $\mathbb{C}^r$. The variety of critical points in \cite{Denham-Garrousian-Schulze} is defined when $\mathcal{A}$ is a central arrangement, so we suppose that this is the case.  In our notation, it is the quotient under the torus action
\[
\widetilde{\mathfrak{X}}(U):=\mathfrak{X}_V(U)/(\mathbb{C}^* \times \mathbf{1}) \subseteq (\mathbb{C}^{r} \times \mathbb{P}^{n-1})/(\mathbb{C}^* \times \mathbf{1}) = \mathbb{P}^{r-1} \times \mathbb{P}^{n-1},
\]
where $V$ is the partial compactification of $U$ by the affine space $\mathbb{C}^r$. Since $\mathcal{A}$ is central, by \cite[Proposition 3.9]{Orlik-Terao} we have
\[
\widetilde{\mathfrak{X}}(U)\subseteq \{u_1+\cdots+u_n=0\}\simeq \mathbb{P}^{r-1} \times \mathbb{P}^{n-2} \subset \mathbb{P}^{r-1} \times \mathbb{P}^{n-1}.
\]

The quotient variety $\widetilde{\mathfrak{X}}(U)$ is indeed the variety of critical points (of a closely related arrangement) in our sense. Consider a decone $\widetilde{\mathcal{A}}$ of $\mathcal{A}$, an affine arrangement obtained by declaring one of the hyperplanes in the projectivization of $\mathcal{A}$ to be the hyperplane at infinity. The number of hyperplanes and the rank of the decone are one less than the corresponding quantities for $\mathcal{A}$. More precisely, we have the following relation between the characteristic polynomials:
\[
\chi_{\widetilde{\mathcal{A}}}(q) = \chi_\mathcal{A}(q)/(q-1).
\]
Let $\widetilde{U}$ be the complement of the decone in $\mathbb{C}^{r-1}$, and take the obvious compactification $\mathbb{P}^{r-1}$ of $\widetilde{U}$. Then the variety of critical points of $\widetilde{U}$ is the subvariety considered above,
\[
\widetilde{\mathfrak{X}}(U) \subseteq \mathbb{P}^{r-1} \times \mathbb{P}^{n-2}.
\]
The reader is invited to compare the formula of Corollary \ref{DGS} with its cohomology version \cite[Theorem 1.1]{Denham-Garrousian-Schulze} for essential central arrangements.
\end{remark}

\subsection{Nonvanishing at the boundary}

Let $H_0,\ldots,H_n$ be the torus-invariant hyperplanes in $\mathbb{P}^n$ defined by the homogeneous coordinates $z_0,\ldots,z_n$. Fix the open embedding
\[
\iota : (\mathbb{C}^*)^n \longrightarrow \mathbb{P}^n,
\]
defined by the ratios $z_1/z_0,\ldots,z_n/z_0$. Let $V$ be the closure of $U$ in $\mathbb{P}^{n}$, and choose a simple normal crossing resolution of singularities
\[
\xymatrix{
\pi^{-1}(U) \ar[r] \ar[d]& \widetilde{V} \ar[d]^\pi &\\
U\ar[r]^\iota&V \ar[r]& \mathbb{P}^{n},
}
\]
where $\pi$ is an isomorphism over $U$, $\widetilde{V}$ is smooth, and $\widetilde{V} \setminus \pi^{-1}(U)$ is a simple normal crossing divisor with the irreducible components $D_1,\ldots,D_k$. Our goal is to show that a sufficiently general differential form on $\widetilde{V}$ with logarithmic poles along $D_j$ has a zero scheme which is a finite set of reduced points in $U$.

Each $f_i$ defines a rational function on $\widetilde{V}$ which is regular on $\pi^{-1}(U)$. We have that
\[
\text{ord}_{D_j} (f_i) \ \text{is}\ \begin{cases} 
\text{positive} & \text{if $\pi(D_j) \nsubseteq H_0 $ and $\pi(D_j) \subseteq H_i$,}\\ 
\text{negative}& \text{if $\pi(D_j) \subseteq H_0 $ and $\pi(D_j) \nsubseteq H_i$.}\\
\end{cases}
\]

\begin{lemma}\label{A}
For each $j$, there is an $i$ such that $\text{ord}_{D_j} (f_i)$ is nonzero.
\end{lemma}

\begin{proof} 
Since $D_j$ is irreducible, each $\pi(D_j)$ is contained in some $H_i$. The assertion follows from the following set-theoretic reasoning:

\begin{enumerate}[1.]
\item If $\pi(D_j) \nsubseteq H_0$, then $\pi(D_j) \subseteq H_i$ for some $i$ because $\bigcup_{i=0}^n H_i = \mathbb{P}^n \setminus (\mathbb{C}^*)^n$. 
\item If $\pi(D_j) \subseteq H_0$, then $\pi(D_j) \nsubseteq H_i$ for some $i$ because $\bigcap_{i=0}^n H_i = \emptyset$.  
\end{enumerate}
\end{proof}

An integral vector $\mathbf{u}=(u_1,\ldots,u_n) \in \mathbb{Z}^n$ defines a rational function on $\widetilde{V}$
\[
\varphi_{\bf u} = \prod_{i=1}^n f_i^{u_i}, 
\]
which is regular on $\pi^{-1}(U)$. Note that for each $j$,
\[
\text{ord}_{D_j}(\varphi_{\bf u})=\sum_{i=1}^n u_i \cdot \text{ord}_{D_j} (f_i).
\]
Combining this with Lemma \ref{A}, we have the following result.

\begin{lemma}\label{middle}
For a sufficiently general $\mathbf{u} \in \mathbb{Z}^n$, $\text{ord}_{D_j}(\varphi_{\bf u})$ is nonzero for all $j=1,\ldots,k$.
\end{lemma}

\noindent More precisely, there are corank-$1$ subgroups $A_1,\ldots,A_k$ of $\mathbb{Z}^n$ such that $\text{ord}_{D_j}(\varphi_{\bf u})$
is nonzero for $\mathbf{u} \in \mathbb{Z}^n \setminus \cup_{j=1}^k A_j$ and $j=1,\ldots,k$.

Consider the sheaf of logarithmic differential one-forms $\Omega_{\widetilde{V}}^1(\log D)$, where $D=D_1+\cdots+D_k$. For the definition and the needed properties of the sheaf of logarithmic differential one-forms, we refer to \cite{Deligne,Saito}. We note that $\Omega_{\widetilde{V}}^1(\log D)$ is a locally free sheaf of rank $r$, and the rational function $\varphi_{\bf u}$ defines a global section
\[
\text{dlog}(\varphi_{\bf u})=\sum_{i=1}^n u_i \cdot \text{dlog}(f_i) \in H^0\big(\widetilde{V}, \Omega_{\widetilde{V}}^1(\log D)\big).
\]

\begin{lemma}\label{nonvanishing}
For a sufficiently general $\mathbf{u} \in \mathbb{Z}^n$, $\text{dlog}(\varphi_{\bf u})$ does not vanish on $D$.
\end{lemma}

\begin{proof} 
Given a point $x \in D$, let $D_1,\ldots,D_l$ be the irreducible components of $D$ containing $x$, and let $g_1,\ldots,g_l$ be local defining equations on a small neighborhood $G$ of $x$. Clearly, $l$ is at least $1$. By replacing $G$ with a smaller neighborhood if necessary, we may assume that $\Omega^1_{\widetilde{V}}(\log D)$ trivializes over $G$, and that
\[
\varphi_{\bf u}= g_1^{o_1} \cdots g_l^{o_l} h \quad \text{where} \ o_j=\text{ord}_{D_j}(\varphi_{\bf u})
\]
for some nonvanishing holomorphic function $h$ on $G$. Over the open set $G$, we have
\[
\text{dlog}(\varphi_{\bf u}) = \Bigg(\sum_{j=1}^l \text{ord}_{D_j}(\varphi_{\bf u}) \cdot \text{dlog}(g_j) \Bigg)+\psi,
\]
where $\psi$ is a regular differential one-form. Since the $\text{dlog}(g_j)$ form part of a free basis of a trivialization of $\Omega_{\widetilde{V}}^1(\log D)$ over $G$, it follows from Lemma \ref{middle} that $\text{dlog}(\varphi_{\bf u})$ does not vanish on $D$ for a sufficiently general $\mathbf{u} \in \mathbb{Z}^n$.
\end{proof}

\subsection{Proof of Theorem \ref{main}}\label{ProofTheoremMain} 

Let $\mathfrak{X}(U)=\mathfrak{X}_{\widetilde{V}}(U)$ be the variety of critical points of $U$ in $\widetilde{V} \times \mathbb{P}(W)$, where $W$ is the complex vector space $M_U \otimes_\mathbb{Z} \mathbb{C}$ as before.
Write $W_{\widetilde{V}}$ for (the sheaf sections of) the trivial vector bundle over $\widetilde{V}$ with the fiber $W$, and consider the homomorphism $\Psi$ defined by evaluation of the logarithmic differential forms
\[
\Psi: W_{\widetilde{V}} \longrightarrow \Omega^1_{\widetilde{V}}(\log D), \qquad (x,\mathbf{u}) \longmapsto \text{dlog}(\varphi_\mathbf{u})(x):=\sum_{i=1}^n u_i \cdot \text{dlog}(f_i)(x).
\]
We do not attempt to give an individual meaning to the multivalued master function $\varphi_\mathbf{u}$ when $\mathbf{u}=(u_1,\ldots,u_n)$ is not integral. Let $\text{pr}_1$ and $\text{pr}_2$ be the two projections from $\widetilde{V} \times \mathbb{P}(W)$, and define the incidence variety of the evaluation by
\[
\mathscr{I}(U) = \mathscr{I}_{\widetilde{V}}(U)=\big\{(x,\mathbf{u}) \in \widetilde{V} \times \mathbb{P}(W) \mid \text{dlog}(\varphi_\mathbf{u})(x)=0 \big\}.
\]
We drop the subscript from $\mathscr{I}_{\widetilde{V}}(U)$ when there is no danger of confusion.
By definition, $\mathfrak{X}(U)$ is the unique irreducible component of $\mathscr{I}(U)$ which dominates $\widetilde{V}$ under $\text{pr}_1$.  

Then, for a sufficiently general $\mathbf{u} \in \mathbb{Z}^n$,

\begin{enumerate}[1.]
\item $\text{pr}_2^{-1}(\mathbf{u}) \cap \mathscr{I}(U)$ is contained in $\mathfrak{X}^\circ(U)$, by Lemma \ref{nonvanishing}, and
\item $\text{pr}_2^{-1}(\mathbf{u}) \cap \mathfrak{X}^\circ(U)$ is a finite set of reduced points, by the Bertini theorem applied to $\text{pr}_2$ on $\mathfrak{X}(U)$ \cite[Theoreme 6.10]{Jouanolou}. 
\end{enumerate}

More precisely, there is a nonempty Zariski open subset $\mathcal{U} \subseteq W$ such that the two assertions are valid for any element in the infinite set $M_U \cap \mathcal{U}$.

It follows that the zero scheme of the section $\text{dlog}(\varphi_\mathbf{u})\in H^0\big(\widetilde{V}, \Omega_{\widetilde{V}}^1(\log D)\big)$,
\[
\text{pr}_1\big(\text{pr}_2^{-1}(\mathbf{u})\big)=\big\{x \in \widetilde{V} \mid \text{dlog}(\varphi_\mathbf{u})(x) =0 \big\},
\]
is a finite set of reduced points in $U$. The smoothness of $\widetilde{V}$ implies that the section $\text{dlog}(\varphi_{\bf u})$ is regular, and the above set represents the homology class $c_r\big(\Omega_{\widetilde{V}}^1(\log D)\big) \cap [\widetilde{V}]$ \cite[Example 3.2.16]{Fulton0}. Since all the $f_i$ are nonvanishing on $U$, we may identify the critical points of $\varphi_{\bf u}$ with the zero scheme of $\text{dlog}(\varphi_{\bf u})$.

Therefore all the critical points of $\varphi_{\bf u}$ are nondegenerate, and the number of critical points is equal to the degree of the top Chern class of $\Omega^1_{\widetilde{V}}(\log D)$.
Finally, from the logarithmic Poincar\'e-Hopf theorem \cite{Kawamata,Norimatsu,Silvotti}, we have
\[
\int_{\widetilde{V}} c_r\big(\Omega_{\widetilde{V}}^1(\log D)\big)=(-1)^r \int_{\widetilde{V}} c_r\big(\Omega_{\widetilde{V}}^1(\log D)^\vee\big) =(-1)^r\chi(U).
\]

\qed

\section{Deletion-restriction for very affine varieties}\label{additivity}

In this section we formulate the deletion-restriction for the characteristic polynomial of a hyperplane arrangement in the very affine setting. The role of the characteristic polynomial will be played by a characteristic class for very affine varieties. 

This point of view is particularly satisfactory for very affine varieties satisfying a genericity condition at infinity, and gives new insights on the positivity of the coefficients of the characteristic polynomial. 
For complements of hyperplane arrangements, we recover the geometric formula for the characteristic polynomial of Denham et al. \cite[Theorem 1.1]{Denham-Garrousian-Schulze}.

Let $U$ be a very affine variety, and let $U_0$ be the hypersurface of $U$ defined by the vanishing of a regular function. The complement $U_1=U\setminus U_0$ is a very affine variety, being a principal affine open subset of a very affine variety.

\begin{definition}
A \emph{triple} of very affine varieties is a collection of the form $(U,U_0,U_1)$.
\end{definition}

 In the language of hyperplane arrangements, $U$ corresponds to the arrangement obtained by deleting the distinguished hyperplane $U_0$ from the arrangement corresponding to $U_1$. For this reason, we call $U$ the deletion of $U_1$ and call $U_0$ the restriction of $U_1$. As in the case of hyperplane arrangements, we have
\[
\dim U_1 = \dim U \quad  \text{and} \quad \dim U_0= \dim U -1.
\]
By the set-theoretic additivity of the topological Euler characteristics of complex algebraic varieties \cite[Section 4.5]{Fulton}, we have
\[
\chi(U)=\chi(U_1)+\chi(U_0).
\]
Therefore, by Theorem \ref{main}, the maximum likelihood degrees of a triple of smooth very affine varieties satisfy an additive formula.
Write $\text{ML}(U)$ for the maximum likelihood degree of a very affine variety $U$, i.e. the number of critical points of a master function of $U$ with sufficiently general exponents.

\begin{corollary}\label{add}
If $(U,U_0,U_1)$ is a triple of smooth very affine varieties, then
\[
\text{ML}(U)=\text{ML}(U_1)-\text{ML}(U_0).
\]
\end{corollary}

It would be interesting to obtain a direct justification of Corollary \ref{add}.

\begin{remark}
Let us call a very affine variety \emph{primitive} if it does not have any deletion. For example, the complement of an essential hyperplane arrangement is primitive if and only if it is the complement of a Boolean arrangement. Note that this is the case exactly when the complement is isomorphic to its intrinsic torus.
A distinguished feature of the deletion-restriction for very affine varieties, when compared to that for hyperplane arrangements, is that there are primitive very affine varieties which are not isomorphic to a torus. These very affine varieties are responsible for the noncombinatorial aspect of the extended theory.
\end{remark}

When $(U,U_0,U_1)$ is a triple of hyperplane arrangement complements, Corollary \ref{add} is the deletion-restriction formula for the M\"obius invariant $\chi_\mathcal{A}(1)$, where $\chi_\mathcal{A}(q)$ is the characteristic polynomial of an affine hyperplane arrangement $\mathcal{A}$. 
The full deletion-restriction formula between the characteristic polynomials can be formulated in terms of the Chern-Schwartz-MacPherson (CSM) class \cite{MacPherson}. Below we give a brief description of the CSM class; Aluffi provides a gentle introduction in \cite{Aluffi}.

Recall that the group of constructible functions $C(X)$ of an algebraic variety $X$ is generated by functions of the form $\mathbf{1}_Z$, where $Z$ is a subvariety of $X$.  If $f: X \to Y$ is a morphism between complex algebraic varieties, then the pushforward of constructible functions is defined by the homomorphism
\[
f_* : C(X) \longrightarrow C(Y), \qquad \mathbf{1}_Z \longmapsto \Big( p \longmapsto \chi\big(f^{-1}(p) \cap Z \big) \Big).
\]
If $X$ is a compact complex manifold, then the characteristic class of $X$ is the Chern class of the tangent bundle $c(TX) \cap [X] \in A_*(X)$, where $A_*(X)$ is the Chow homology group of $X$ (see \cite{Fulton0}). A generalization is provided by the Chern-Schwartz-MacPherson class, whose existence was once a conjecture of Deligne and Grothendieck. For a construction with emphasis on smooth and possibly noncompact varieties, see \cite{Aluffi2}.

Let $C$ be the functor of constructible functions from the category of complex algebraic varieties (with proper morphisms) to the category of abelian groups.

\begin{definition}
The \emph{CSM class} is the unique natural transformation
\[
c_{SM} : C \longrightarrow A_*
\]
such that $c_{SM}(\mathbf{1}_X) = c(TX) \cap [X]$ when $X$ is smooth and complete.
\end{definition}

The uniqueness follows from the naturality, the resolution of singularities, and the normalization for smooth and complete varieties. The CSM class satisfies the inclusion-exclusion relation
\[
c_{SM}(\mathbf{1}_{U \cup U'})=c_{SM}(\mathbf{1}_{U})+c_{SM}(\mathbf{1}_{U'})-c_{SM}(\mathbf{1}_{U \cap U'})
\]
and captures the Euler characteristic as its degree
\[
\chi(U) = \int c_{SM}(\mathbf{1}_U).
\]
When $U$ is the complement of an arrangement of hyperplanes $\mathcal{A}$ in $\mathbb{C}^r$, the CSM class of $U$ is the characteristic polynomial $\chi_\mathcal{A}(q)$. For the definition of the characteristic polynomial of an affine arrangement, see  \cite[Definition 2.52]{Orlik-Terao2}.

\begin{theorem}\label{characteristic}
Let $\mathbb{P}^r$ be the compactification of $U$ defined by the hyperplane at infinity $\mathbb{P}^r \setminus \mathbb{C}^r$. Then
\[
c_{SM}(\mathbf{1}_U)=\sum_{i=0}^{r} (-1)^i v_i \hspace{0.7mm} [\mathbb{P}^{r-i}]  \in A_*(\mathbb{P}^r),
\]
where
\[
\chi_\mathcal{A}(q+1)=\sum_{i=0}^{r} (-1)^i v_i \hspace{0.7mm} q^{r-i}.
\]
\end{theorem}

This is because the recursive formula for a triple of arrangement complements
\[
c_{SM}(\mathbf{1}_{U_1})=c_{SM}(\mathbf{1}_{U}-\mathbf{1}_{U_0})=c_{SM}(\mathbf{1}_{U})-c_{SM}(\mathbf{1}_{U_0}),
\]
agrees with the usual deletion-restriction formula
\[
\chi_{\mathcal{A}_1}(q+1)=\chi_\mathcal{A}(q+1)-\chi_{\mathcal{A}_0}(q+1)
\]
(see \cite[Theorem 2.56]{Orlik-Terao2}).
The induction is on the dimension and on the number of hyperplanes. The case of no hyperplanes involves a direct computation of $c_{SM}(\mathbf{1}_{\mathbb{C}^r})$ by the inclusion-exclusion formula, and the case of dimension $1$ is a special case of the equality
\[
\chi(U) = \int c_{SM}(\mathbf{1}_U)=\chi_\mathcal{A}(1).
\]
See \cite[Theorem 1.2]{Aluffi3} and also \cite[Remark 26]{Huh1}.

Our goal is to relate the variety of critical points to the CSM class. If we restrict our attention to the degree of the CSM class, then the relation recovers the conclusion stated in Varchenko's conjecture for the Euler characteristic. We prove this for a class of very affine varieties satisfying a genericity condition at infinity. 

The genericity condition is commonly expressed using the language of tropical compactifications. If $U$ is a subvariety of an algebraic torus $\mathbb{T}$, then we consider the closures $\overline{U}$ of $U$ in various (not necessarily complete) normal toric varieties $X$ of $\mathbb{T}$. The closure $\overline{U}$ is complete if and only if the support of the fan of $X$ contains the tropicalization of $U$ \cite[Proposition 2.3]{Tevelev}. We say that $\overline{U}$ is a \emph{tropical compactification} of $U$ if it is complete and the multiplication map
\[
m: \mathbb{T} \times \overline{U} \longrightarrow X, \quad (t,x) \longmapsto tx
\]
is flat and surjective. Tropical compactifications exist, and they are obtained from toric varieties $X$ defined by sufficiently fine fan structures on the tropicalization of $U$ \cite[Section 2]{Tevelev}. 

\begin{definition}\label{schon}
We say that $U$ is \emph{sch\"on} if the multiplication is smooth for some tropical compactification of $U$.
\end{definition}

Equivalently, $U$ is sch\"on if the multiplication is smooth for every tropical compactification of $U$ \cite[Theorem 1.4]{Tevelev}.

\begin{remark}
There are two classes of sch\"on very affine varieties that are of particular interest. The first is the class of complements of essential hyperplane arrangements, and the second is the class of nondegenerate hypersurfaces \cite{Tevelev}. 
What we need from the sch\"on hypothesis is the existence of a simple normal crossings compactification which admits sufficiently many logarithmic differential one-forms. For arrangement complements, such a compactification is provided by the wonderful compactification of De Concini and Procesi \cite{DP}. For nondegenerate hypersurfaces, and more generally for nondegenerate complete intersections, the needed compactification has been constructed by Khovanskii \cite{Hovanskii}.
\end{remark}

Let $U \subseteq \mathbb{T}_U$ be a very affine variety of dimension $r$, where $\mathbb{T}_U$ is the intrinsic torus of Section \ref{GaussMap}. Let $V$ be the closure of $U$ in $\mathbb{P}^n$, where $\mathbb{P}^n$ is a fixed toric compactification of $\mathbb{T}_U$. We follow Section \ref{VarietyofCriticalPoints} and define the variety of critical points
\[
\mathfrak{X}(U) \subseteq V \times \mathbb{P}(W) \subseteq \mathbb{P}^n \times \mathbb{P}^{n-1} \quad \text{where} \ W=M_U \otimes_\mathbb{Z} \mathbb{C}.
\]

\begin{theorem}\label{proCSM2}
Suppose that $U$ is sch\"on and not isomorphic to a torus. Then
\[
\big[\mathfrak{X}(U)\big] =
\sum_{i=0}^{r} v_i \big[ \mathbb{P}^{r-i} \times \mathbb{P}^{n-1-r+i}\big] \in A_{*}(\mathbb{P}^n \times \mathbb{P}^{n-1}),
\]
where
\[
c_{SM}(\mathbf{1}_U)=\sum_{i=0}^{r} (-1)^i v_i \hspace{0.7mm} [\mathbb{P}^{r-i}] \in A_*(\mathbb{P}^n).
\]
\end{theorem}

\begin{proof}

We prove a slightly more general statement. Let $V$ be a compactification of $U$ obtained by taking the closure in a toric variety $X$ of $\mathbb{T}_U$. We will prove the equality
\[
c_{SM}({\mathbf 1}_U) = \sum_{i=0}^r (-1)^i \text{pr}_{1*} \big[ \text{pr}_2^{-1}(\mathbb{P}^{r-i}) \cap \mathfrak{X}(U) \big] \in A_*(V),
\]
where $\text{pr}_1$ and $\text{pr}_2$ are the two projections from $V \times \mathbb{P}(W)$ and $\mathbb{P}^{r-i}$ is a sufficiently general linear subspace of $\mathbb{P}(W)$ of the indicated dimension. The projection formula shows that this implies the stated version when $X=\mathbb{P}^n$.

If $U$ is a sch\"on very affine variety, then there is a tropical compactification of $U$ which has a simple normal crossings boundary divisor. More precisely, there is a smooth toric variety $\widetilde{X}$ of $\mathbb{T}_U$, obtained by taking a sufficiently fine fan structure on the tropicalization of $U$, such that the closure of $U$ in $\widetilde{X}$ is a smooth and complete variety $\widetilde{V}$ with the simple normal crossings divisor $\widetilde{V} \setminus U$ \cite[Proof of Theorem 2.5]{Hacking}.

By taking a further subdivision of the fan of $\widetilde{X}$ if necessary, we may assume that there is a toric morphism $\widetilde{X} \to X$ preserving $\mathbb{T}_U$. By the functoriality of the CSM class, we have
\[
A_*(\widetilde{V}) \longrightarrow A_*(V), \qquad c_{SM}(\mathbf{1}_U) \longmapsto c_{SM}(\mathbf{1}_U).
\]
Note also that
\[
A_*\big(\widetilde{V} \times \mathbb{P}(W)\big) \longrightarrow  A_*\big(V \times \mathbb{P}(W)\big), \qquad \big[\mathfrak{X}_{\widetilde{V}}(U)\big] \longmapsto \big[\mathfrak{X}_{V}(U)\big].
\]
By the projection formula, the problem is reduced to the case where $\widetilde{X}=X$. 

In this case, we have the following exact sequence induced by the restriction $\mathcal{O}_X|_V \to \mathcal{O}_V$:
\[
\xymatrix{
0 \ar[r] & N_{V/X}^\vee \ar[r] & \Omega^1_X(\log X \setminus \mathbb{T}_U)|_V \ar[r] & \Omega^1_{V}(\log V \setminus U) \ar[r] & 0.
}
\]
Note that the middle term is isomorphic to the trivial vector bundle $W_V$ over $V$ with the fiber $W$ \cite[Section 4.3]{Fulton}. Under this identification, the restriction of the differential one-forms is the evaluation map of Section \ref{ProofTheoremMain},
\[
\xymatrix{
0 \ar[r] &\ker \Psi \ar[r] & W_{V}  \ar[r]^{\Psi \hspace{9mm}}& \Omega^1_{V}(\log V \setminus U).
}
\]
It follows that the latter sequence is exact, and the projectivization of the kernel 
\[
\mathscr{I}(U) = \big\{(x,\mathbf{u}) \in V \times \mathbb{P}(W) \mid \text{dlog}(\varphi_\mathbf{u})(x)=0 \big\}
\]
coincides with the variety of critical points $\mathfrak{X}(U)$. Since the pullback of $\mathcal{O}_{\mathbb{P}(W)}(1)$ to $\mathfrak{X}(U)$ is the canonical line bundle $\mathcal{O}_{\mathbb{P}\big(N_{V/X}^\vee \big)}(1)$, we have
\begin{eqnarray*}
\sum_{i=0}^r \text{pr}_{1*} \big[ \text{pr}_2^{-1}(\mathbb{P}^{r-i}) \cap \mathfrak{X}(U) \big] &=& \text{pr}_{1*} \Big( \sum_{i=0}^r c_1\big(\text{pr}_2^*\ \mathcal{O}_{\mathbb{P}(W)}(1)\big)^{n-1-r+i} \cap \big[\mathfrak{X}(U) \big]\Big) \\
&=& s\big(N_{V/X}^\vee\big) \cap \big[V\big]  
\ = \ c\big(\Omega^1_{V}(\log V \setminus U)\big) \cap \big[V\big].
\end{eqnarray*}
Here $n$ is the dimension of $W$, and the last equality is the Whitney sum formula. Now the assertion follows from the fact that the CSM class of a smooth variety is the Chern class of the logarithmic tangent bundle; that is,
\[
c_{SM}(\mathbf{1}_U)= c\big(\Omega^1_{V}(\log V \setminus U)^\vee\big) \cap \big[V\big].
\]
This follows from a construction of the CSM class which is most natural from the point of view of this paper \cite[Section 4]{Aluffi2}. For precursors, see \cite[Theorem 1]{Aluffi0} and also \cite[Proposition 15.3]{Goresky-Pardon}.
\end{proof}

\begin{remark}
In a refined form, the above proof shows that the equality
\[
c_{SM}({\mathbf 1}_U) = \sum_{i=0}^r (-1)^i \text{pr}_{1*} \big[ \text{pr}_2^{-1}(\mathbb{P}^{r-i}) \cap \mathfrak{X}(U) \big]
\]
holds in the $\mathbb{T}$-equivariant proChow group $\widehat{A}^{\hspace{0.5mm} \mathbb{T}}_*(\mathbb{T}_U)$ of \cite{Aluffi2.5,Aluffi2}. This removes the dependence on the compactification from Theorem \ref{proCSM}. It should not be expected that the equality holds in the ordinary proChow group of $\mathbb{T}_U$.
\end{remark}

\begin{remark}
Theorem \ref{proCSM2} may fail to hold for a smooth very affine variety. As an example, consider a smooth hypersurface $U$ in $(\mathbb{C}^*)^3$ whose closure in $\mathbb{P}^3$ has a node at a torus orbit of codimension $1$. A direct computation on a simple normal crossings compactification of $U$ shows that the incidence variety $\mathscr{I}(U)$  has a two-dimensional component other than $\mathfrak{X}(U)$, and hence the classes of $\mathscr{I}(U)$ and $\mathfrak{X}(U)$ are different in $A_2(\mathbb{P}^3 \times \mathbb{P}^2)$.
\end{remark}

For complements of hyperplane arrangements, Theorem \ref{proCSM2} gives a geometric formula for the characteristic polynomial \cite[Theorem 1.1]{Denham-Garrousian-Schulze}. Let $U$ be the complement of an arrangement of $n$ distinct hyperplanes 
\[
\mathcal{A} = \{f_1=0\} \cup \cdots \cup \{f_n=0\} \subset \mathbb{C}^r.
\]
Then $U$ is a very affine variety if and only if $\mathcal{A}$ is an essential arrangement, meaning that the lowest-dimensional intersections of the hyperplanes of $\mathcal{A}$ are isolated points. Indeed, taking one of the isolated points as the origin of $\mathbb{C}^r$ and choosing $r$ linearly independent hyperplanes intersecting at that point reveals $U$ to be a principal affine open subset of $(\mathbb{C}^*)^r$. For the converse, write $U$ as a product $U' \times \mathbb{C}^k$, where $U'$ is the complement of an essential arrangement, and note that the affine line $\mathbb{C}$ does not admit a closed embedding into an algebraic torus.

Suppose from now on that $\mathcal{A}$ is an essential arrangement. Then $\mathcal{A}$ is a Boolean arrangement if and only if $U$ is isomorphic to an algebraic torus. The equations of the hyperplanes define a closed embedding 
\[
f: U \longrightarrow (\mathbb{C}^*)^n \simeq \mathbb{T}_U, \qquad f=(f_1,\ldots,f_n).
\]
The indicated isomorphism follows from the linear independence of the $f_i$ in $M_U$. We fix the open embedding $(\mathbb{C}^*)^n \subset \mathbb{P}^n$ defined by the ratios of homogeneous coordinates $z_1/z_0,\ldots,z_n/z_0$. Then the closure of $U$ in $\mathbb{P}^n$ is a linear subspace $\mathbb{P}^r \subset \mathbb{P}^n$. In this setting, combining Theorems \ref{characteristic} and \ref{proCSM2} gives the following statement.

\begin{corollary}\label{DGS}
Suppose that $\mathcal{A}$ is not a Boolean arrangement. Then
\[
\big[\mathfrak{X}(U)\big] =
\sum_{i=0}^{r} v_i \big[ \mathbb{P}^{r-i} \times \mathbb{P}^{n-1-r+i}\big] \in A_{*}(\mathbb{P}^r \times \mathbb{P}^{n-1}),
\]
where
\[
\chi_\mathcal{A}(q+1)=\sum_{i=0}^{r} (-1)^i v_i \hspace{0.7mm} q^{r-i}.
\]
\end{corollary}

\begin{remark}\label{logconcave}
A sequence $e_0,\ldots,e_r$ of integers is said to be \emph{log-concave} if $e_{i-1}e_{i+1}\le e_i^2$ for all $i$,
and it is said to have \emph{no internal zeros} if the indices of the nonzero elements are consecutive integers. Write a homology class $\xi \in A_k(\mathbb{P}^n \times \mathbb{P}^m)$ as the linear combination
\[
\xi = \sum_i e_i \big[\mathbb{P}^{k-i} \times \mathbb{P}^i \big].
\]
It can be shown that some multiple of $\xi$ is the fundamental class of an irreducible subvariety if and only if the $e_i$ form a log-concave sequence of nonnegative integers with no internal zeros \cite[Theorem 21]{Huh1}.

Therefore, by Theorem \ref{proCSM2}, the $v_i$ of a sch\"on very affine variety form a log-concave sequence of nonnegative integers with no internal zeros. In particular, the coefficients of $\chi_\mathcal{A}(q+1)$ form a sequence with the three properties. This strengthens the previous result that the coefficients of $\chi_\mathcal{A}(q)$ form a log-concave sequence \cite[Theorem 3]{Huh1}, and answers several questions on sequences associated to a matroid, for matroids representable over a field of characteristic zero:

\begin{enumerate}
\item Read's conjecture predicts that the coefficients of the chromatic polynomial of a graph form a unimodal sequence \cite{Read}. This follows from the log-concavity of $\chi_\mathcal{A}(q)$ when $\mathcal{A}$ is the graphic arrangement of a given graph \cite{Huh1}.
\item Hoggar's conjecture predicts that the coefficients of the chromatic polynomial of a graph form a strictly log-concave sequence  \cite{Hoggar}.  This follows from the log-concavity of $\chi_\mathcal{A}(q+1)$ when $\mathcal{A}$ is the graphic arrangement of a given graph \cite{Huh2}.
\item Welsh's conjecture predicts that the $f$-vector of a matroid complex forms a unimodal sequence \cite{Welsh}. This follows from the log-concavity of $\chi_\mathcal{A}(q)$ when $\mathcal{A}$ is an arrangement corresponding to the cofree extension of a given matroid \cite{Lenz}.
\item Dawson's conjecture predicts that the $h$-vector of a matroid complex forms a log-concave sequence \cite{Dawson}. This follows from the log-concavity of $\chi_\mathcal{A}(q+1)$ when $\mathcal{A}$ is an arrangement corresponding to the cofree extension of a given matroid  \cite{Huh2}.
\end{enumerate}
For details on the derivation of the above variations, see \cite{Huh2}.
\end{remark}

\begin{remark}
The characteristic class approach to Varchenko's conjecture and the generalized deletion-restriction have been pioneered by Damon for nonlinear arrangements on smooth complete intersections \cite{Damon1,Damon2}. In fact, it can be shown that Damon's higher multiplicities are the degrees of the CSM class of the arrangement complement. See \cite{Huh1} for the connection between the two, for nonlinear arrangements on a projective space.

The CSM point of view successfully deals with several problems considered by Damon in \cite{Damon1,Damon2}. In particular, an affirmative answer to the conjecture in \cite[Remark 2.6]{Damon2} follows from the product formula
\[
c_{SM}(\mathbf{1}_{U_1} \otimes \mathbf{1}_{U_2})=c_{SM}(\mathbf{1}_{U_1}) \otimes c_{SM}(\mathbf{1}_{U_2}) \in A_*(\overline{U_1} \times \overline{U_2}). 
\]
The above is a refinement of the product formula of Kwieci\'nski in \cite{Kwiecinski}, and can be viewed as a generalization of the product formula for the Euler characteristic,
\[
\chi(U_1 \times U_2) = \chi(U_1) \chi(U_2);
\]
see \cite[Th\'eor\`eme 4.1]{Aluffi2.5}.
\end{remark}

\section{Nondegenerate hypersurfaces}\label{Kushnirenko}

A nondegenerate Laurent polynomial defines a hypersurface in an algebraic torus which admits a tropical compactification with a simple normal crossings boundary divisor \cite[Section 2]{Hovanskii}. 
A sufficiently general Laurent polynomial with the given Newton polytope is nondegenerate, and the corresponding hypersurface is sch\"on.

We show that the variety of critical points of a hypersurface defined by a nondegenerate Laurent polynomial is controlled by the Newton polytope. This gives a formula for the CSM class in terms of the Newton polytope, which specializes to Kouchnirenko's theorem equating the Euler characteristic with the signed volume of the Newton polytope \cite{Kouchnirenko}. 

Let $g$ be a nonzero Laurent polynomial in $n$ variables
\[
g=\sum_\mathbf{u} c_\mathbf{u} \mathbf{x}^\mathbf{u} \in \mathbb{C}[x_1^{\pm 1},\ldots,x_n^{\pm 1}].
\]
We are interested in the CSM class of the very affine variety
\[
U=\{g=0\} \subseteq (\mathbb{C}^*)^n.
\]
The \emph{Newton polytope} of $g$, denoted by $\Delta_g$, is  the convex hull of exponents $\mathbf{u} \in \mathbb{Z}^n$ with nonzero coefficient $c_\mathbf{u}$.
Write $g_\sigma$ for the Laurent polynomial made up of those terms of $g$ which lie in a face $\sigma$ of the Newton polytope. We say that $g$ is \emph{nondegenerate} if $dg_\sigma$ is nonvanishing on $\{g_\sigma=0\}$ for every face of its Newton polytope.

We follow the convention of \cite[Chapter 7]{Cox-Little-Oshea} and write $\text{MV}_n$ for the $n$-dimensional mixed volume. For example, the $n$-dimensional standard simplex $\Delta$ in $\mathbb{R}^n$ has normalized volume $1$.

In view of later applications to projective hypersurfaces, we state our result for a fixed compactification $(\mathbb{C}^*)^n \subset \mathbb{P}^n$, where the open embedding is defined by the ratios of homogeneous coordinates $z_1/z_0,\ldots,z_n/z_0$. The formulation for other toric compactifications and the extension to complete intersections are left to the interested reader.  

\begin{theorem}\label{Kush2}
Let $g$ be a nonzero Laurent polynomial in $n=r+1$ variables with
\[
c_{SM}(\mathbf{1}_{U}) = \sum_{i=0}^{r} (-1)^{i} v_i \hspace{0.7mm} [\mathbb{P}^{r-i}] \in A_*(\mathbb{P}^n).
\]
If $g$ is nondegenerate, then
\[
v_i=\text{MV}_n(\underbrace{\Delta,\ldots,\Delta}_{r-i},\underbrace{\Delta_g,\ldots,\Delta_g}_{i+1}) \quad \text{for} \ i=0,\ldots,r.
\]
In particular, the maximum likelihood degree of $U$ is equal to the normalized volume
\[
v_r=(-1)^{r} \int c_{SM}(\mathbf{1}_{U})= \text{Volume}(\Delta_g).
\]
\end{theorem}

Since the degree of $c_{SM}(\mathbf{1}_U)$ is the Euler characteristic $\chi(U)$, this recovers Kouchnirenko's theorem \cite[Th\'eor\`eme IV]{Kouchnirenko}.

\begin{proof}
In any case, $c_{SM}(\mathbf{1}_{(\mathbb{C}^*)^n})$ is the fundamental class of the ambient smooth and complete toric variety. Therefore we may assume that $U$ is nonempty and compute the CSM class of $(\mathbb{C}^*)^n\setminus U$ instead. 
 
Fix a sufficiently fine subdivison of the fan of $\mathbb{P}^n$ on which the support function of $\Delta_g$ is piecewise linear. We may assume that the corresponding toric variety $X$ is smooth and the closure $V$ of $U$ in $X$ has simple normal crossings with $D:=X \setminus (\mathbb{C}^*)^n$ \cite[Section 2]{Hovanskii}. Note in this case that
\[
c_{SM}(\mathbf{1}_{(\mathbb{C}^*)^n }) -c_{SM}(\mathbf{1}_U)
= c\big(\Omega^1_{X}(\log D \cup V)^\vee\big) \cap [X] \in A_*(X).
\]
In order to compute the right-hand side, we use the Poincar\'e-Leray residue map
\[
\text{r\'es}: \Omega^1_{X}(\log D\cup V) \longrightarrow \mathcal{O}_V, \qquad \eta \cdot \text{dlog}(z)+\psi \longmapsto \eta |_V,
\]
where $z$ is a local defining equation for $V$ and $\psi$ is a rational differential one-form which does not have poles along $V$. 
The restriction of $\eta$ on $V$ is uniquely and globally determined, and in particular it does not depend on the choice of $z$.
Note that the residue map fits into the exact sequence
\[
\xymatrix{0 \ar[r] &\Omega^1_{X}(\log D) \ar[r] &\Omega^1_{X}(\log D\cup V)\ar[r]^{\qquad \text{r\'es}} & \mathcal{O}_{V} \ar[r] & 0.
}
\]
Since $\Omega^1_{X}(\log D)$ is a trivial vector bundle, the Whitney sum formula shows that
\[
c\big(\Omega^1_{X}(\log D \cup V)^\vee\big) \cap [X]
=\sum_{i=0}^n (-1)^i c_1\big(\mathcal{O}_X(V)\big)^i \cap [X].
\]
Therefore, by the projection formula applied to the birational map $X \to \mathbb{P}^n$, we have
\[
c_{SM}(\mathbf{1}_{(\mathbb{C}^*)^n})-c_{SM}(\mathbf{1}_{U}) = \sum_{i=0}^n (-1)^i  (H^{n-i} \cdot V^i) [\mathbb{P}^{n-i}] \in A_*(\mathbb{P}^n),
\]
where $H$ is the pullback of a hyperplane in $\mathbb{P}^n$. 

It remains to show that the intersection product $(H^{n-i} \cdot V^i) $ is equal to the mixed volume $\text{MV}_n(\underbrace{\Delta,\ldots,\Delta}_{n-i},\underbrace{\Delta_g,\ldots,\Delta_g}_i)$.  For this one computes the divisor of the rational function $g$ on $X$,
\[
\text{Div}(g)=V+\sum_\rho \text{ord}_{D_\rho}(g) D_\rho=V+\sum_\rho \psi(u_\rho) D_\rho.
\]
Here $u_\rho$ is the primitive ray generator of a ray $\rho$, $D_\rho$ is the torus-invariant prime divisor of $X$ corresponding to $\rho$, and $\psi$ is the support function of $\Delta_g$. It follows that $V$ is linearly equivalent to the torus-invariant divisor
\[
V_\infty:=-\sum_\rho \psi(u_\rho) D_\rho.
\]
Therefore, by \cite[Section 5.4]{Fulton}, we have
\[
 (H^{n-i} \cdot V^i)= (H^{n-i} \cdot V_\infty^i)=\text{MV}_n(\underbrace{\Delta,\ldots,\Delta}_{n-i},\underbrace{\Delta_g,\ldots,\Delta_g}_i).
\]
\end{proof}

\begin{remark}
In this case, log-concavity of the $v_i$ discussed in Remark \ref{logconcave} is the Alexandrov-Fenchel inequality on the mixed volume of convex bodies \cite[Section 6.3]{Schneider}.
\end{remark}

\begin{remark}
Let $\mathscr{V}$ be a subvariety of the projective space with the homogeneous coordinates $p_0,p_1,\ldots,p_n$. 
In the statistical setting of \cite{Hosten-Khetan-Sturmfels}, one computes the maximum likelihood degree of the very affine variety
\[
\mathscr{U}:=\big\{x \in \mathscr{V} \mid p_0 p_1\cdots p_n (p_0+p_1+\cdots+p_n) \neq 0\big\} \subset \mathbb{P}^{n}.
\]
Suppose that $\mathscr{V}$ is the closure of the hypersurface in $(\mathbb{C}^*)^n$ defined by a sufficiently general Laurent polynomial with the given Newton polytope.
In the notation used in the proof of Theorem \ref{Kush2}, we have the residue exact sequence
\[
\xymatrix{0 \ar[r] &\Omega^1_{X}(\log D\cup V) \ar[r] &\Omega^1_{X}(\log D\cup V \cup H)\ar[r]^{\qquad \quad \text{r\'es}} & \mathcal{O}_{H} \ar[r] & 0,
}
\]
where $H$ is the pullback of the hypersurface $\{p_0+p_1+\cdots+p_n=0\}$ in $\mathbb{P}^n$. Therefore, by the Whitney sum formula,
\[
\text{ML}(\mathscr{U})=v_1+\cdots+v_{n}.
\]
When the Newton polytope is the $d$-th multiple of the standard simplex $\Delta$, we have
\[
\text{ML}(\mathscr{U})=d+\cdots+d^n=d\cdot \frac{d^n-1}{d-1}.
\]
This is the formula of \cite[Theorem 6]{Hosten-Khetan-Sturmfels}.
\end{remark}

The fact that the whole CSM class of a nondegenerate hypersurface in $(\mathbb{C}^*)^n$ is determined by the Newton polytope (and not just the topological Euler characteristic) has applications not covered by Kouchnirenko's theorem. In the remainder of this section we show how Theorem \ref{Kush2} can be applied to obtain results on the geometry of projective hypersurfaces. In particular, we give an explicit formula for the degree of the gradient map of a nondegenerate homogeneous polynomial in terms of its Newton polytope.

Let $h$ be a nonconstant homogeneous polynomial in $n+1$ variables
\[
h=\sum_\mathbf{u} c_\mathbf{u} \mathbf{z}^\mathbf{u} \in \mathbb{C}[z_0,\ldots,z_n].
\]
The gradient map of $h$ is the rational map
\[
\text{grad}(h): \mathbb{P}^n \dashrightarrow \mathbb{P}^n, \qquad \text{grad}(h)=\Bigg(\frac{\partial h}{\partial z_0}: \cdots : \frac{\partial h}{\partial z_n}\Bigg).
\]
The study of the gradient map of a homogeneous polynomial is one of the central topics in classical projective geometry. We refer to \cite{CRS} and \cite{DolgachevBook} for a historical introduction.

\begin{definition}
Let $\Gamma_h \subset \mathbb{P}^n \times \mathbb{P}^n$ be the closure of the graph of the gradient map of $h$. We define the integers $\mu^0(h),\ldots,\mu^n(h)$ by the formula
\[
[\Gamma_h] = \sum_{i=0}^n \mu^i (h)[\mathbb{P}^{n-i} \times \mathbb{P}^i ] \in A_n(\mathbb{P}^n \times \mathbb{P}^n).
\]
\end{definition}

For any nonconstant $h$:
\begin{itemize}
\item $\mu^0(h)$ is $1$;
\item $\mu^1(h)$ is one less than the reduced degree of $h$;
\item $\mu^n(h)$ is the degree of the gradient map of $h$; and 
\item in general, $\mu^i(h)$ is the number of $i$-dimensional cells in a CW-model of $D(h)$ (see \cite[Theorem 9]{Huh1}).
\end{itemize}

The following theorem of Aluffi relates the CSM class to the gradient map of a homogeneous polynomial \cite[Theorem 2.1]{Aluffi1}.

\begin{theorem}
Let $D(h)$ be the smooth affine variety
\[
D(h)=\{h \neq 0\} \subset \mathbb{P}^n.
\]
The CSM class of $D(h)$ is given by the formula
\[
c_{SM}(\mathbf{1}_{D(h)}) = \sum_{i=0}^n (-1)^i \mu^i(h) H^{i} (1+H)^{n-i} \in A_*(\mathbb{P}^n), 
\]
where $H$ is the class of a hyperplane in $\mathbb{P}^n$.
\end{theorem}

Our goal is to show that almost all homogeneous polynomials $h$ with the given Newton polytope $\Delta_h$ define a projective hypersurface with the same Milnor numbers $\mu^i(h)$, and there is an explicit formula for computing these numbers from $\Delta_h$. This proves an assertion left unjustified in \cite[Example 17]{Huh1}.

Let us call a convex polytope $\Delta_*$ in $\mathbb{R}^{n+1}$ \emph{homogeneous} if it lies in an affine hyperplane perpendicular to the vector $(1,\ldots,1)$. 
We normalize the $n$-dimensional mixed volume for homogeneous polytopes so that the $n$-dimensional standard simplex $\Delta_\circ$ in $\mathbb{R}^{n+1}$ has unit volume. 

For a homogeneous convex polytope $\Delta_*$ in $\mathbb{R}^{n+1}$, define
\[
m_i(\Delta_*):=\text{MV}_n(\underbrace{\Delta_\circ,\ldots,\Delta_\circ}_{n-i},\underbrace{\Delta_*,\ldots,\Delta_*}_i) \quad \text{for} \ i=0,\ldots,n.
\]
Note that $m_0$ vanishes exactly when $\Delta_*$ is empty.
We extend \cite[D\'efinition 1.7]{Kouchnirenko} and define the \emph{mixed Newton numbers} of $\Delta_*$ by the formula
\[
\nu_i(\Delta_*)=V_{n,i}-V_{n-1,i-1}+\cdots+(-1)^i V_{n-i,0} \quad \text{for} \ i=0,\ldots,n,
\]
where $V_{k,l}$ is the sum of the $k$-dimensional mixed volumes $m_l$ of the intersections of $\Delta_*$ with all possible $(k+1)$-dimensional coordinate planes. 

A homogeneous polynomial is said to be \emph{nondegenerate} if its dehomogenization with respect to a variable is nondegenerate as a Laurent polynomial. Almost all homogeneous polynomials with the given Newton polytope are nondegenerate. 

\begin{corollary}\label{ProjectiveDegree}
If $h$ is a nondegenerate homogeneous polynomial in $(n+1)$ variables, then
\[
c_{SM}(\mathbf{1}_{D(h)}) = \sum_{i=0}^n (-1)^i \nu_i(\Delta_h) [\mathbb{P}^{n-i}] \in A_*(\mathbb{P}^n).
\]
In particular, the degree of the gradient map of $h$ is the sum
\[
\mu^n(h) = \sum_{i=0}^n \nu_i(\Delta_h).
\]
\end{corollary}

\begin{proof}
Write $I$ for a proper subset of $\{0,1,\ldots,n\}$. We use the decomposition
\[
D(h)=\coprod_I D(h) \cap \mathbb{T}_I \quad \text{where} \ \mathbb{T}_I=\big\{(z_0:\cdots:z_n) \in \mathbb{P}^n \mid z_i=0 \ \text{for} \ i \in I\big\}.
\]
Let $h_I$ be the homogeneous polynomial obtained from $h$ by setting the variable $z_i$ to zero for $i \in I$. Since every face of the Newton polytope of $h_I$ is a face of the Newton polytope of $h$, the nondegeneracy of $h$ implies nondegeneracy of $h_I$. Therefore Theorem \ref{Kush2} computes the CSM class of $D(h) \cap \mathbb{T}_I$ in terms of $\Delta_{h_I}$ for every $I$. Collecting terms of the same dimension in the resulting expressions, we have
\[
c_{SM}(\mathbf{1}_{D(h)}) = \sum_I c_{SM}(\mathbf{1}_{D(h) \cap \mathbb{T}_I})=\sum_{i=0}^n (-1)^i \nu_i(\Delta_h) [\mathbb{P}^{n-i}]. 
\]
\end{proof}

This shows that many delicate examples discovered in classical projective geometry have a rather simple combinatorial origin. As an example, we consider Hesse's claim, which states in effect that the degree of the gradient map of a homogeneous polynomial is zero if and only if the corresponding projective hypersurface is a cone \cite{Hesse1,Hesse2}. Hesse's claim is true up to dimension $3$, but counterexamples were found by Gordan and Noether in dimension $4$ (see \cite{Gordan-Noether}). Here is a nondegenerate counterexample with given degree $d \ge 3$:
\[
h=z_3^{d-1} z_0+ z_3^{d-2}z_4z_1+z_4^{d-1} z_2.
\]
The actual values of the coefficients are irrelevant, so long as they are nonzero; all the polynomials define the same hypersurface up to a linear change of coordinates. 

\begin{remark}
Note that we do not have a hypothesis corresponding to \emph{commode} of \cite[D\'efinition 1.5]{Kouchnirenko}. No positivity property of the Newton numbers
should be expected in this generality. 
\end{remark}

\begin{example}
A homaloidal polynomial is a homogeneous polynomial $h$ whose gradient map is a birational transformation of $\mathbb{P}^n$. The property of being homaloidal depends only on the set $\{h=0\} \subseteq \mathbb{P}^n$ (see \cite{Dimca-Papadima}).
We check Corollary \ref{ProjectiveDegree} for three nondegenerate homaloidal plane curves.
\begin{enumerate}[1.]
\item The nonsingular conic $\{h=x^2+y^2+z^2=0\}\subset \mathbb{P}^2$: In this case,
\begin{eqnarray*}
\mu^2(h)&=&V_{2,0}+(V_{2,1}-V_{1,0})+(V_{2,2}-V_{1,1}+V_{0,0})\\
&=& 1+(2-3)+(4-6+3)\\
&=& 1.
\end{eqnarray*}
\item The union of three nonconcurrent lines $\{h=xyz=0\} \subset \mathbb{P}^2$: In this case,
\begin{eqnarray*}
\mu^2(h)&=&V_{2,0}+(V_{2,1}-V_{1,0})+(V_{2,2}-V_{1,1}+V_{0,0})\\
&=& 1+(0-0)+(0-0+0)\\
&=& 1.
\end{eqnarray*}
\item The union of the conic and its tangent $\{h=zy^2+x^2y=0\}\subset \mathbb{P}^2$: In this case,
\begin{eqnarray*}
\mu^2(h)&=&V_{2,0}+(V_{2,1}-V_{1,0})+(V_{2,2}-V_{1,1}+V_{0,0})\\
&=& 1+(2-2)+(0-0+0)\\
&=& 1.
\end{eqnarray*}
\end{enumerate}
Dolgachev proved in \cite{Dolgachev} that any homaloidal plane curve is equal to one of the above, up to a linear change of coordinates of $\mathbb{P}^2$. It would be interesting to extend Dolgachev's list and, as a first step, classify nondegenerate homaloidal surfaces in $\mathbb{P}^3$.
\end{example}

\begin{example}\label{homaloidal}
The modern study of homaloidal polynomials began with the work of Ein and Shepherd-Barron \cite{Ein-Shepherd-Barron}, as well as that of Etingof et al. \cite{EKP}, who observed that the relative invariant of a regular prehomogeneous vector space is homaloidal. These homaloidal polynomials have degree bounded in terms of the ambient dimension. 
In contrast to the case of relative invariants and that of plane curves, Ciliberto et al. showed that there are irreducible homaloidal hypersurfaces of given degree $d \ge 2n-3$ and ambient dimension $n \ge 3$ (see \cite[Theorem 3.13]{CRS}). Their construction relies on an ingenious argument concerning projective duals of rational scroll surfaces. 

Corollary \ref{ProjectiveDegree} provides a cheap way of producing projective hypersurfaces with interesting numerical properties (perhaps with the aid of a computer) and, in particular, homaloidal polynomials. Once the correct form of the equation is guessed, it is often not hard to check by hand that the example has the desired property. We demonstrate this by showing that there are irreducible homaloidal hypersurfaces of given degree $d \ge 3$ and ambient dimension $n \ge 3$. 

Among many explicit examples of this type, we present a particularly transparent one:
\[
h=z_0z_n^{d-1}+z_1z_2z_n^{d-2}+z_2^d+(z_3^2+z_4^2+\cdots+z_{n-1}^2)z_n^{d-2}.
\]
The actual values of the coefficients are irrelevant, so long as they are nonzero; all the polynomials define the same hypersurface up to a linear change of coordinates. By the Bertini theorem, $h$ is nondegenerate and irreducible. The system of equations
\[
y_0=\frac{\partial h}{\partial z_0}, \ldots, y_n=\frac{\partial h}{\partial z_n}
\]
is triangular (of De Jonqui\`eres type), 
defining an isomorphism between the affine spaces $z_n\neq 0$ and $y_0\neq 0$. Therefore $h$ is homaloidal for any $d \ge 3$ and $n \ge  3$. We note that Fassarella and Medeiros constructed (necessarily reducible) homaloidal hypersurfaces with given degree $d \ge 3$ and ambient dimension $n \ge 3$, using an idea based on the CSM class \cite[Example 5.1]{Fassarella-Medeiros}.
\end{example}

\begin{acknowledgements}
The author is grateful to Bernd Sturmfels for asking the questions which initiated this paper. He thanks Mircea Musta\c t\u a for useful discussions throughout the preparation of the manuscript. He also thanks Nero Budur, Eric Katz, and Sam Payne for helpful comments. He thanks the anonymous referee for reading the manuscript carefully and offering valuable suggestions.
\end{acknowledgements}

\end{document}